\documentclass[12pt]{article}

\usepackage{amssymb}
\usepackage{amsmath}
\usepackage{amsbsy}
\usepackage{amscd}
\usepackage{amsfonts}
\usepackage{amsthm}
\usepackage{mathrsfs}
\usepackage{verbatim}
\usepackage[colorlinks]{hyperref}
\usepackage{fullpage}
\usepackage{mathdots}
\usepackage{graphicx,subfigure}
\usepackage[english]{babel}
\usepackage[utf8]{inputenc}
\usepackage{tikz}
\usepackage{enumitem}

\usepackage{authblk}

\usepackage{mathtext}

\usepackage{caption}

\usetikzlibrary{backgrounds,fit, matrix}
\usetikzlibrary{positioning}
\usetikzlibrary{calc,through,chains}
\usetikzlibrary{arrows,shapes,snakes,automata, petri}

\newtheorem{Theorem}[equation]{Theorem}
\newtheorem{Corollary}[equation]{Corollary}
\newtheorem{Lemma}[equation]{Lemma}
\newtheorem{Proposition}[equation]{Proposition}

\theoremstyle{definition}
\newtheorem{Definition}[equation]{Definition}

\newtheorem{Remark}[equation]{Remark}

\numberwithin{equation}{section}
\numberwithin{figure}{section}

\newcommand{\PP}{{\mathbb P}}
\newcommand{\C}{{\mathbb C}}

\newcommand{\mc}[1]{\mathcal{#1}}

\newcommand{\mt}[1]{\text{#1}}

\newcommand{\mbf}[1]{\mathbf{#1}}

\newcommand{\Pic}[1]{\text{Pic}(#1)}
\newcommand{\Comp}[1]{\mbf{\Delta} \setminus #1}

\begin{document}

\title{Toroidal Schubert Varieties}

\author{Mahir Bilen Can\footnote{\hbox{Corresponding author: Email: mahirbilencan@gmail.com, Phone: 1(504) 862 3448, Fax: 1(504) 865 5063}}, 
Reuven Hodges, and Venkatramani Lakshmibai}

\maketitle

\begin{abstract}

Levi subgroup actions on Schubert varieties are studied. 
In the case of partial flag varieties, the horospherical actions are determined.
This leads to a characterization of the toroidal 
and horospherical partial flag varieties with Picard number 1. 
In the more general case, we provide a set of necessary conditions for the action of a Levi
subgroup on a Schubert variety to be toroidal. The singular locus of a (co)minuscule Schubert variety is shown 
to contain all the $L_{max}$-stable Schubert subvarieties, where $L_{max}$ is the standard Levi subgroup 
of the maximal parabolic which acts on the Schubert variety by left multiplication.
In type A, the effect of the Billey-Postnikov decomposition on 
toroidal Schubert varieties is obtained. 

\vspace{.5cm}

\end{abstract}

\normalsize

\section{Introduction}

This paper was first motivated by the following two questions.
\begin{itemize}
\item Which Schubert varieties are toric varieties? 
\item More generally, which Schubert varieties are spherical varieties?
\end{itemize}
On the one hand, it has been known for some time that,
except for a few trivial cases, such as the projective 
space, in general, Schubert varieties are not toric varieties. 
Nevertheless, one knows that every Schubert 
variety has a degeneration to a 
toric variety, see~\cite{Caldero}.
On the other hand, less is known  
with regard to the second question. 
Let us briefly describe 
our current state of knowledge.
Let $G$ be a connected 
reductive complex algebraic
group and let $P$ be a 
parabolic subgroup of $G$. 
We fix a Borel subgroup
$B$ of $G$ that is contained in $P$. 

Clearly, a Schubert variety 
$X$ in $G/P$ is not 
$G$-stable unless 
the equality $X=G/P$ holds true. 
Therefore, we look for the reductive subgroups 
of $G$ which act on $X$. 
Some natural candidates are 
given by the Levi subgroups of 
the stabilizer subgroup $Q:=\mt{Stab}_G(X)$.   
At one extreme, in the Grassmann
variety of $k$ dimensional subspaces of 
$\C^n$, denoted by $\mbf{Gr}(k,n)$,
we know the complete classification 
of the Schubert varieties 
which are spherical with respect to the
action of a Levi subgroup of  
$\mbf{GL}_n=\mbf{GL}_n(\C)$, see~\cite{HodgesLakshmibai}.
At the other extreme, if we assume 
that $X$ is a smooth Schubert variety 
in a partial flag variety $G/P$, where 
$G$ does not have any $\text{G}_2$-factors,
then we have by~\cite{CanHodges} that 
$X$ is a spherical $L_{max}$-variety, 
$L_{max}$ being the Levi factor of the stabilizer 
subgroup in $G$ of $X$.

By definition, a spherical $G$-variety $X$ 
is called {\em toroidal} if in $X$ every $B$-stable 
irreducible divisor which contains 
a $G$-orbit is $G$-stable. 
It turns out that such a
variety $X$ is `locally toric' in the sense
that it contains a toric subvariety $Z$, 
for some torus $T_Z$,  
such that the $G$-orbit 
structure of $X$ is completely 
determined by the $T_Z$-orbit 
structure of $Z$. This remark 
will be made more precise in the sequel.
Regarding this definition, our purpose in the present paper is two-fold;
\begin{enumerate}
\item we characterize the partial flag varieties that are ``horospherical'' and may or may not be toroidal;
\item we initiate a program for determining the Schubert varieties which are toroidal;
\end{enumerate}
Here, horospherical means that the stabilizer of a point from the open orbit
contains a maximal unipotent subgroup. 
We should mention that in this paper we focus only on the actions
of Levi subgroups on Schubert varieties, so, one can 
ask for the analogous results for the other reductive subgroups 
of $G$. See~\cite{AvdeevPetukhov, AvdeevPetukhov2, Brion01}. 
We leave this investigation to a future paper.

Now we are ready to state our main results and at the same time 
give a brief overview of our paper.

In our preliminaries section, we set up our notation and provide 
the basic definitions and results regarding toroidal and regular $G$-varieties.
In Section~\ref{S:3}, we recall Pasquier's work from~\cite{Pasquier} 
on smooth projective horospherical varieties of Picard number 1.
Then we record our first main result,
Theorem~\ref{T:horospherical}, where we classify the partial flag varieties 
on which a Levi subgroup acts horospherically. 
This allows us, in Section~\ref{S:4}, to characterize the horospherical and toroidal actions of Levi subgroups on 
partial flag varieties which have Picard number 1. 

In Section~\ref{S:LeviStable}, 
Proposition~\ref{P:LorbitCriterion} provides a general method for determining if a Schubert variety contains an orbit of a fixed Levi subgroup. Then in the subsequent section, we specialize to the type A Grassmannian, and give precise combinatorial criterion in Proposition~\ref{P:maxConditionHeads} for Schubert divisors to be stable under the action of a fixed Levi subgroup. These culminate in our second main result, Corollary~\ref{C:toroidalGrassNesc}, where we provide a set of necessary (though not sufficient) conditions for a Schubert variety in the Grassmannian to be toroidal.
In Section~\ref{S:Singularlocus}, we focus on the singularities 
of a Schubert variety in a partial flag variety $G/Q$,
where $Q$ is a (co)minuscule parabolic subgroup 
in a simple algebraic group.  
We show that the singular locus of a Schubert
variety $X_{wQ}$ contains all $L_{max}$-stable Schubert subvarieties 
of $X_{wQ}$ (Proposition~\ref{P:singularlocus}), where
$L_{max}$ is the standard Levi subgroup of the stabilizer subgroup $P:=\mt{Stab}_G(X_{wQ})$. 
Hence $X_{wQ}$ has no $L_{max}$-stable Schubert divisors. 
The final section of our paper is about the Billey-Postnikov 
decomposition of a permutation $w$ such that $X_{w\mbf{P}}$ 
is a smooth toroidal Schubert $\mbf{L}$-variety in some partial flag 
variety $\mbf{GL}_n/\mbf{P}$. (The Billey-Postnikov 
decomposition, that we abbreviate to BP-decomposition,
is a certain parabolic decomposition of $w$. 
We will explain this notion in more detail in the sequel.)
Under this assumption, we state our final main result, Theorem~\ref{T:divisors to divisors1}; 
we prove that if $w=vu$ is the BP decomposition 
of $w$ with the corresponding projection 
$\pi : X_{w \mathbf{P}} \rightarrow X_{v\mathbf{Q}}$, 
where $\mathbf{Q}$ is a maximal parabolic subgroup of 
$\mathbf{GL}_n$, then both of the 
Schubert varieties $X_{v\mathbf{Q}}$ and $X_{u\mathbf{P}}$
are smooth toroidal Schubert $\mbf{L}$-varieties. In combination with our results from Sections~\ref{S:LeviStable} and~\ref{S:ToroidalinGrassmann} the above provides necessary conditions for a Schubert variety in any type A (partial) flag variety to be toroidal.

\vspace{.5cm}

\textbf{Acknowledgements.} 
The first author is partially supported by a grant from the Louisiana Board of Regents. 
We are grateful to the referee whose comments helped us to see and fix our mistakes;
this greatly improved the quality of our paper.

\section{Preliminaries}
\label{S:Preliminaries}

We begin with fixing our basic notation.
$$
\begin{array}{lcl}
G &: & \text{ a connected reductive group defined over $\C$;}\\
T &: & \text{ a maximal torus in $G$;}\\
W &: & \text{ the Weyl group of $(G,T)$;}\\
B &: & \text{ a Borel subgroup of $G$ such that $T\subset B$;}\\
U &: & \text{ the unipotent radical of $B$ (so, $B=T\ltimes U$);}\\
S &: & \text{ the Coxeter generators of $W$ determined by $B$;}\\
C_w &: & \text{ the $B$-orbit through $wB/B$ ($w\in W$) in $G/B$;}\\
\ell &: & \text{ the length function defined by $w\mapsto \dim C_w$ ($w\in W$);}\\
X_{w} &: & \text{ the Zariski closure of $C_w$ ($w\in W$) in $G/B$;}\\
P,Q &:& \text{ parabolic subgroups in $G$;}\\
p_{{\scriptscriptstyle P,Q}} &: & \text{ the canonical projection 
$p_{{\scriptscriptstyle P,Q}}:G/P\rightarrow G/Q$
if $P\subset Q \subset G$;}\\
W_Q &: & \text{ the unique subgroup of $W$ corresponding to $Q$ where $B\subset Q\subset G$;}\\
W^Q &: & \text{ the minimal length left coset representatives
of $W_Q$ in $W$;}\\
X_{wQ} &: & \text{ the image of $X_w$ in $G/Q$ under $p_{{\scriptscriptstyle B,Q}}$;}\\
R_u(H) &: & \text{ the unipotent radical of an algebraic group $H$.}
\end{array}
$$
A variety of the form $X_{wQ}$, where $Q$ is any parabolic subgroup  
in $G$ and $w\in W^Q$ is called a \emph{Schubert variety}.
Its ambient homogenous space $G/Q$ is called a \emph{(partial) flag variety}. 
The \emph{Bruhat-Chevalley} order on $W^Q$ is defined so that for $v$ and $w$ from $W^Q$ 
we have 
$$
w \leq v \iff X_{wQ} \subseteq X_{vQ}.
$$

Let $w_0$ denote the element with maximal length in $W$
and let $w_{0,Q}$ denote the element with maximal length in $W_Q$.


If $G$ is the general linear group $\mbf{GL}_n$ and and $Q$ is a maximal parabolic subgroup in $G$ whose simple roots omit $\alpha_k$,
then $G/Q$ is isomorphic to the \emph{Grassmann variety} $\mbf{Gr}(k,n)$.

A parabolic subgroup $Q$ from $G$ is said to be \emph{standard} with respect to $B$,
or, \emph{standard} if $B\subseteq Q$. 
Let us assume that $T\subseteq Q$. 
A Levi subgroup $L$ of $Q$ 
is called \emph{standard} with respect to $T$, or, \emph{standard} if it contains $T$.

Let $\mathbb{T}$ be the $n$ dimensional complex torus, $\mathbb{T}:=(\C^*)^n$. 
A complex normal algebraic $\mathbb{T}$-variety $X$ is called a {\em toric variety} if $X$ admits an
algebraic group action $\mathbb{T}\times X \rightarrow X$ with an open $\mathbb{T}$-orbit. 
Let us mention in passing that, being an abelian group, 
$\mathbb{T}$ is a Borel subgroup of itself, that is to say, a maximal connected
solvable subgroup.
More generally, let $G$ be a complex connected reductive group and 
$X$ be a normal complex algebraic variety 
with an algebraic action $G\times X\rightarrow X$ 
such that a Borel subgroup $B$ of $G$ has an open orbit in $X$. 
In this case, $X$ is called a {\em spherical $G$-variety}. Clearly, 
a toric variety is a spherical variety. 
A closed subgroup 
$H$ of $G$ is said to be a {\em spherical subgroup} 
if $G/H$ is a spherical $G$-variety. 
Note that any parabolic subgroup $Q$ 
is a spherical subgroup. This is a consequence 
of the Bruhat-Chevalley decomposition of $G$.

\vspace{.25cm}

Let $H$ be a spherical subgroup of $G$ and let $X$ be a $G$-variety with an open $G$-orbit isomorphic to
 $G/H$, hence $X$ is a spherical $G$-variety.
The set of $B$-invariant prime divisors in $G/H$ will be denoted 
by $\mc{D}$. 
Let $Y$ be a $G$-orbit in $X$ and define 
$$
\mc{D}_Y(X):= \{ D\in \mc{D} :\ Y\subseteq \overline{D}\}
$$
and set 
$$
\mc{D}(X) := \bigcup_{Y \text{ is a $G$-orbit in $X$}} \mc{D}_Y(X).
$$
These sets of divisors are important in the 
classification of spherical varieties. 

\begin{Definition}\label{D:colorandtoroidal}
A {\em color} of $X$ is a $B$-stable prime divisor 
that is not $G$-stable. 
A spherical $G$-variety $X$ is called {\em toroidal} if $X$ has no colors
containing a $G$ orbit. In particular, $X$ is {\em toroidal} if 
$\mc{D}(X)$ is empty. In this case, we say that the action of $G$ is toroidal. 
\end{Definition}

We will provide an alternative characterization
of the toroidal varieties in the next proposition.
First we fix some notation.
Let $X$ be a spherical $G$-variety and let $P=P_X$ 
denote the stabilizer in $G$ of the set $\Delta_X$ which is  
defined by 
\begin{align*}
\Delta_X = \bigcup_{D\in \mc{D}} \overline{D}.
\end{align*}
Then $P$ is a parabolic subgroup.

\begin{Proposition}\label{P:BrionSpherique}
We preserve the notation from the previous paragraph.
In this case, 
$X$ is toroidal if and only if  
there exists a Levi subgroup $L$ of $P$ 
and a closed subvariety $Z$ of $X\setminus \Delta_X$ 
which is stable under $L$ such that 
\begin{itemize}
\item the map
$
R_u(P) \times Z \rightarrow X\setminus \Delta_X 
$
is an isomorphism;
\item the group $[L,L]$ 
acts trivially on $Z$, which is a toric variety 
for the torus $L/[L,L]$. 
Furthermore, every $G$-orbit 
meets $Z$ along a unique $L$-orbit.  
\end{itemize}
\end{Proposition}
\begin{proof}
See~\cite[Proposition 1]{BrionSpherique}.
\end{proof}

Let $Y$ be a closed subvariety of $X$. 
The {\em normal cone to $X$ in $Y$}, denoted by $C_X(Y)$, is the cone over $X$ 
defined by the graded sheaf

Let $X$ be a $G$-variety such that $G$ has an open orbit,
denoted by $\Omega$. In this case, 
$X$ is called a {\em regular $G$-variety} if it satisfies the following 
three conditions:
\begin{enumerate}
\item[(1)] The closure of every $G$-orbit is smooth.
\item[(2)] For any orbit closure $Y\neq X$, $Y$ is the 
intersection of the orbit closures of codimension one 
containing $Y$.
\item[(3)] The isotropy group of any point $x\in X$ 
has a dense orbit in the normal space to the orbit $G\cdot x$ in $X$.
\end{enumerate}

The normal space that is mentioned in (3) is the quotient vector space 
\hbox{$T_p X/ T_p (G\cdot p)$}, 
where $T_pX$ is the tangent space of $X$ at $p$, and $T_p (G\cdot p)$ is the subspace of $T_pX$ 
that is isomorphic to the tangent space at $p$ of the homogeneous space $G\cdot p$. 
It follows from (1) that a regular $G$-variety is smooth.
Let $x$ be a point from $\Omega$ and let $H$ denote 
the isotropy group of $x$ in $G$. Then $\Omega \cong G/H$. 
The {\em boundary divisor of $X$} is the set $X\setminus \Omega$. 
It is a union of finitely many $G$-stable prime divisors, $X_1,\dots, X_\ell$.
In other words,
$$
X\setminus \Omega = \bigcup_{i=1}^\ell X_i.
$$

Regular $G$-varieties are independently introduced by
Bifet, De Concini, and Procesi in~\cite{BDP90} and by Ginsburg in~\cite{Ginsburg}.
Wonderful compactifications of symmetric varieties as well as 
all smooth toric varieties are examples of regular $G$-varieties. 

\begin{Lemma}\label{L:BB}
Let $X$ be a smooth and complete $G$-variety. 
Then $X$ is a regular $G$-variety if and only if $X$ is a toroidal spherical $G$-variety. 
\end{Lemma}
\begin{proof}
See~\cite[Proposition 2.2.1]{BienBrion}.
\end{proof}

Let us also mention that Brion has computed the intersection numbers of 
the smooth toroidal completions (hence, of the regular completions) of the spherical homogeneous spaces,
see~\cite{Brion98}.

\begin{Remark}
It follows from the definition of a regular $G$-variety that a $G$-orbit closure in a regular $G$-variety is a regular $G$-variety. 
\end{Remark}



We continue with setting up more notational conventions. 
1) ``$X$ is a homogeneous $H$-variety.'' means that 
$H$ acts on $X$ morphically, and $X$ is a homogeneous variety of the form $G/K$, 
where $G$ is an algebraic group, which may or may not be equal to $H$.
In the case when $G$ equals $H$, we will say that ``$X$ is an $H$-homogeneous variety.'' 
2) When we refer to \emph{the} action of a subgroup of $G$ on a partial flag variety $G/Q$ 
(or a subvariety of $G/Q$) we are always referring to the action by left multiplication.

Let $Q$ and $P$ denote two parabolic subgroups in $G$. 
\begin{Lemma}\label{L:1}
Let $L$ and $L'$ be two Levi subgroups in $P$. 
If the action of $L$ on $G/Q$ is spherical, then so is the action of $L'$.
\end{Lemma}
\begin{proof}
This follows from the fact that $L$ and $L'$ are conjugates 
by an element from the unipotent radical of $P$. 
\end{proof}

A closely related observation is the following. 

\begin{Lemma}\label{L:2}
Let $L$ be a Levi subgroup of $P$, and let $g$ be an element from $G$. 
If the action of $L$ on $G/Q$ is spherical, then so is the action of 
any Levi subgroup of $gPg^{-1}$. 
\end{Lemma}
\begin{proof}
Clearly, $gLg^{-1}$ is a Levi subgroup of $gPg^{-1}$. 
Composing the action of $L$ by conjugation by $g^{-1}$ shows that 
$gL g^{-1}$ acts spherically on $G/Q$. The rest of the proof follows from Lemma~\ref{L:1}.
\end{proof}

In view of Lemmas~\ref{L:1} and~\ref{L:2}, we can always assume that both of the subgroups $P$ and $Q$ 
are standard parabolic subgroups with $B\subseteq P\cap Q$. 
We will make another reduction argument.
Let $L$ be a Levi subgroup of $P$ such that the action of $L$ on $G/Q$ is spherical. 
Note that the action of the center of $G$ on $G/Q$ is trivial, and the action of $L$ on $G/Q$ 
is induced from that of $G$. 
So, whenever it is convenient for our purposes, without loss of generality, 
we will assume that $G$ is a connected semisimple group.

By~\cite[Lemma 5.3]{AvdeevPetukhov}, $G/Q$ is a spherical $L$-variety 
if and only if the diagonal action of $G$ on $G/P\times G/Q$ is spherical. 
As Littelmann showed in~\cite{Littelmann}, the study of such diagonal actions is very useful in representation theory. 
For $G\in \{\mbf{SL}_n,\mbf{Sp}_n\}$, Magyar, Zelevinsky, Weyman~\cite{MWZ1, MWZ2} classified all
parabolic subgroups $P_1,\dots, P_k\subset G$ such that the variety $\prod_{i=1}^r G/P_i$ has only finitely 
many diagonal $G$-orbits. For $P_1=B$, this classification amounts to the classification of 
spherical varieties of the form $\prod_{i=2}^r G/P_i$ with respect to the diagonal $G$-actions. 
By Stembridge's work~\cite{Stembridge}, we know the complete 
classification of the diagonal spherical actions of $G$ on the products of the form $G/P\times G/Q$.
\begin{Remark}
Let us assume that both of $P$ and $Q$ are properly contained in $G$. Then 
by using Stembridge's list we see that a Levi subgroup $L$ of $P$ can act spherically on $G/Q$ 
only if the simple factors of the semisimple part of $G$ avoid the types $\text{E}_8,\text{F}_4$, and $\text{G}_2$. 
\end{Remark}



\section{Horospherical Flag Varieties}\label{S:3}

In this section, we review briefly the work of Pasquier 
on the smooth projective horospherical varieties of Picard number 1. 
Then we will present our classification result on the horospherical partial flag varieties.
Let us start with the definition of horospherical varieties. 
\begin{Definition}
Let $X$ be a $G$-variety. $X$ is called {\em horospherical} if the stabilizer of a point in a 
general position contains a maximal unipotent subgroup of $G$. 
\end{Definition}

Let $G$ be a connected reductive group, let $T$ be a maximal torus in $G$,
and let $B$ be a Borel subgroup of $G$ such that $T\subset B$. 
We will denote by $\Gamma$ the Dynkin diagram of $(G,B,T)$. 

Let $S$ denote the set of simple roots that is determined by $(G,B,T)$,
and let $\Lambda$ and $\Lambda^+$ denote the character group of $B$ and 
the subsemigroup of dominant characters, respectively. 

For $\alpha \in S$, we denote by $\omega_\alpha$ the 
corresponding fundamental weight in $\Lambda^+$.
If the elements of $S$ are ordered $\alpha_1,\dots, \alpha_\ell$, then 
we write $\omega_1,\dots, \omega_\ell$ for the corresponding fundamental weights.
We will denote by $P(\omega_\alpha)$ the maximal parabolic subgroup of $G$
such that $B\subset P(\omega_\alpha)$ and $\omega_\alpha$ is a character of $P(\omega_\alpha)$.

Let $G/H$ be a horospherical homogeneous space, and let $P$ 
denote the normalizer of $H$ in $G$. Then $G/H\to G/P$ is a torus bundle of rank $n$. 
Without loss of generality we assume that $H$ contains $U$, the unipotent radical of $B$, 
therefore, we assume that $B\subset P$. Then there is a 
subset $I \subset S$ such that $P= B W_I B$, where $W_I$ is the parabolic subgroup 
of $W$ that is generated by the simple reflections corresponding to the elements of $I$.  
This $I$ is also given by 
$
I= \{ \alpha \in S:\ \text{ $\omega_i$ is not a character of $P$}\}.
$

\begin{Definition}
Let $G/H$ be a horospherical homogeneous space. 
If it exists, we denote by $X^1$ the unique smooth projective embedding of $G/H$ with 
Picard number 1. If $X^1$ exists, we will call $G/H$ a {\em special homogeneous space}, 
and we will call $S\setminus I$ {\em the set of colors of $X^1$}.
\end{Definition}

According to Pasquier~\cite[Section 1.2.2]{Pasquier}, if $X^1$ is not a projective space,
then the horospherical subgroup $H$ is given by the kernel of the character 
$
\omega_\alpha - \omega_\beta + \chi : P(\omega_\alpha) \cap P(\omega_\beta) \longrightarrow \C^*,
$
where $\alpha$ and $\beta$ are the two colors for $X^1$, and $\chi$ is a character of the 
center of $G$. 
Moreover, assuming that $G$ is a simple group, Pasquier shows in~\cite[Section 1.3]{Pasquier} that if $X^1$ is a 
homogeneous space but not a projective space, then the triplet 
$(\Gamma,\alpha,\beta)$ is given by the one of the following:
\begin{enumerate}[label=(\alph*)]
\item \label{Pas:a} $(\textrm{A}_m, \alpha_1,\alpha_m)$, with $m\geq 2$;
\item \label{Pas:b} $(\textrm{A}_m, \alpha_i,\alpha_{i+1})$, with $m\geq 3$ and $i\in \{1,\dots, m-1\}$;
\item \label{Pas:c} $(\textrm{D}_m, \alpha_{m-1},\alpha_{m})$, with $m\geq 4$.
\end{enumerate}
It is shown in~\cite[Propositions 1.8--1.10]{Pasquier} that $X^1$ is isomorphic to the closure
in $\PP(V(\omega_\alpha)\oplus V(\omega_\beta))$ of $G\cdot [v_\alpha+v_\beta]$, 
where $V(\omega_\alpha)$
(resp. $V(\omega_\beta)$) stands for the irreducible $G$-module with highest weight $\omega_\alpha$
(resp. the irreducible $G$-module with highest weight $\omega_\beta$). 
Here, $v_\alpha$ and $v_\beta$ are the corresponding highest weight vectors.
The homogeneous spaces for the cases~\ref{Pas:a},~\ref{Pas:b}, and~\ref{Pas:c}
are given by $\mbf{SO}_{2m+2}/ P(\omega_1)$,
$\mbf{Gr}(i+1,m+2)$ ($1\leq i \leq m-1$), and $\mbf{Spin}(2m+1)/P(\omega_m)$, respectively.

\begin{Remark}\label{R:twoorbits}
Notice that in all of these three cases the two closed $G$-orbits in $X^1$ are given by 
$G\cdot [v_\alpha]$ and $G\cdot [v_\beta]$. 
To see this, we observe that $[v_\alpha]$ lies in the closure of 
the orbit $G\cdot [v_\alpha + v_\beta]$. Indeed, this is easy to see
by acting on $[v_\alpha+v_\beta]$ 
with an appropriate 1-parameter subgroup, and by taking limit.
Let us demonstrate this procedure explicitly in the case of \ref{Pas:a},
the other cases being similar. 

In this case, $X^1$ turns out to be isomorphic to a quadric hypersurface 
in \hbox{$\PP(\C^{m+1}\oplus (\C^{m+1})^*)$}, whose dimension is $2m$. 
The two fundamental irreducible representations are given by $V(\omega_1) \cong \C^{m+1}$ and 
$V(\omega_m) \cong (\C^{m+1})^*$, where the latter is the dual of the first one.
Let $v_1$ and $v_m = v_1^*$ be two highest weight vectors from $V(\omega_1)$ and $V(\omega_m)$,
respectively. Then we can choose $\lambda : \C^* \to T$ so that 
$$
z\cdot v_1:= \lambda (z) \cdot v_1 = z v_1\  \text{ and } \ 
z\cdot v_m := \lambda(z) \cdot v_m =- \frac{1}{z} v_m
$$
for all $z\in \C^*$. Thus,
$z \cdot [v_1+ v_m ] = [z v_1 - 1/z v_m]$. But observe that  
$$
\lim_{z\to \infty}  [z v_1 - 1/z v_m] = \lim_{z\to \infty}  [z v_1] = [v_1]. 
$$
Thus, $[v_1] \in \overline{G\cdot [v_1+v_m]}$.
In particular, $G\cdot [v_1] \subset X^1$. Since $[v_1]$ is the highest weight vector,
its stabilizer in $G$ is a parabolic subgroup, hence the orbit is closed. 
The containment $G\cdot [v_m] \subset X^1$ is seen in a similar way. 
In particular, there are only three $G$-orbits in $X^1$; there are two closed $G$-orbits and an open orbit. 
This is true for the homogeneous space $X^1$ of cases \ref{Pas:b} and \ref{Pas:c} as well. 



\end{Remark}

Let us assume that $G$ is a connected semisimple group, $P$ and $Q$ are two standard parabolic subgroups in $G$,
$L$ is a Levi subgroup of $P$, and let $X$ denote the partial flag variety $G/Q$. 
We assume that $X$ is an $L$-spherical variety. 
 Under these assumptions, the main result of this section is the following: 
\begin{Theorem}\label{T:horospherical}
If $L$ acts horospherically on $X$, then one of the following holds true:
\begin{enumerate}
\item $X$ is a projective space;
\item $X$ is a homogeneous $L$-variety; 
\item $X$ is an $Q'/Q$-bundle over a homogeneous $L$-variety of the form $\prod_{i=1}^r X_i$, where 
$Q'$ is a maximal parabolic subgroup containing $Q$, and $X_i$ ($1\leq i \leq r$) is one of the varieties 
\begin{enumerate}
\item $\mbf{SO}_{2m+2}/ P(\omega_1)$ with $m\geq 2$,
\item $\mbf{Gr}(i+1,m+2)$ ($1\leq i \leq m-1$) with $m\geq 3$, 
\item $\mbf{Spin}(2m+1)/P(\omega_m)$ with $m\geq 4$.
\end{enumerate}
Furthermore, in this case, the simple factors of
$L$ are either of type $\text{A}_m$ ($m\geq 2$), or of type $\text{D}_m$ ($m\geq 4$).
\end{enumerate}
\end{Theorem}

\begin{proof}
Let us assume that $X$ is neither a projective space nor a homogeneous $L$-variety. 
Then we will show that $X$ is a homogeneous $L$-variety. 
To this end, let $Q'$ denote a maximal parabolic subgroup containing $Q$, and let $X'$ denote the generalized Grassmannian
$G/Q'$. Then the canonical morphism $p_{Q,Q'} : X\to X'$ is a $G$-equivariant, hence $L$-equivariant, projection onto $X'$. 
In fact, $p_{Q,Q'}$ is a fiber bundle morphism with fibers isomorphic to $Q'/Q$.  

It is now easy to see that, since the $L$-action on $X$ is horospherical, the action of $L$ on $G/Q'$ is horospherical as well.
We now restrict our attention to the simple factors of $L$, so, without loss of generality we assume that $L$ is simple. 
In this case, since $Q'$ is a maximal parabolic subgroup of $G$, we know also that the Picard number of $X'$ is 1. 
Therefore, we are in one of the three situations,~\ref{Pas:a},~\ref{Pas:b}, or~\ref{Pas:c}, 
that are described by Pasquier.
In particular, we see that $L$ is of type $\text{A}_m$ ($m\geq 2$) or of type $\text{D}_m$ ($m\geq 4$). 
This finishes the proof of our theorem. 
\end{proof}

\section{Some Toroidal Flag Varieties}\label{S:4}

We proceed with our previous notation: 
$G$ is a connected reductive group, $P$ and $Q$ are two parabolic subgroups in $G$,
$L$ is a Levi subgroup of $P$, $X:=G/Q$. 
Furthermore, we assume that $L$ acts spherically on $X$. 
In addition, we will assume that $X$ is a horospherical $L$-variety. 
We fix a Borel subgroup $B$ of $G$ such that $B\subset Q\cap P$.
Let $T$ be a maximal torus in $B$, and let $W$ denote the associated Weyl group of $(G,T)$. 
Then $B$ determines a generating system of simple reflections $S$ in $W$. 
We will denote by $W_Q$ the Weyl group of $(L_Q,T)$, where $L_Q$ is a Levi subgroup 
in $Q$ such that $T\subset L_Q$. In this case, $W_Q$ is a subgroup of $W$. 
Let $w$ denote the element with maximal length in $W$,
and let $w_{0,Q}$ denote the element of maximal length in $W_Q$. 
Let $B_{L}^-$ be a Borel subgroup in $L$ such that $B_{L}^-id Q$ is open in $G/Q$,
and let $B_{L}$ denote the Borel subgroup in $L$ such that $B_{L} := L \cap B$.
Then $B_{L} \subset \textrm{Stab}_{L}(idQ)$. 
\vspace{.25cm}

Recall that toroidal $G$-varieties are those spherical varieties whose colors do not contain any $G$-orbit. 
Equivalently, they are the spherical $G$-varieties $X$ such that 
every $B$-stable divisor of $X$ which contains a $G$-orbit is $G$-stable.

The main result of this section is the following.

\begin{Theorem}\label{T:Pic1}
We assume that $X$ is a horospherical $L$-variety of the form $X=G/Q$. 
In addition, we assume that the Picard number of $X$ is 1.
Then $X$ is toroidal if and only if one of the following conditions are satisfied:
\begin{enumerate}
\item $X$ is a projective space;
\item $X$ is an $L$-homogeneous variety.  
\end{enumerate}
\end{Theorem}

\begin{proof}
First of all, let us note that under our hypothesis, $X$ is the special $L$-equivariant embedding 
of a homogeneous space $L/J$, where $J$ is a horospherical subgroup. 
In other words, in Pasquier's notation, we have $X = X^1$. 

Now, the implication ($\Leftarrow$) is clear, so we will show that if $X$ is toroidal, then
either $X$ is a projective space, or $X$ is homogeneous for the action of $L$.
Towards a contradiction, let us assume that $X$ is toroidal but $X$ is neither a projective 
space, nor it is an $L$-homogeneous variety. 
Thus, according to Pasquier, we are in one the three situations \ref{Pas:a}, \ref{Pas:b}, or \ref{Pas:c} 
that are described in Section~\ref{S:3}.
In other words, $L/J \cong L\cdot [v_\alpha + v_\beta]$ and $X^1$ is the 
closure of the latter orbit in $\PP(V(\omega_\alpha) \oplus V(\omega_\beta))$. 
We already observed in Remark~\ref{R:twoorbits} that 
$\PP(\omega_\alpha)$ and $\PP(\omega_\beta)$ are the two closed $L$-orbits in $X^1=G/Q$.

We now observe an additional fact that the minimal $L$-orbit in $X^1=G/Q$ is a Schubert variety. 
Indeed, 
\[
LQ/Q \cong L /L\cap Q \cong P/P\cap Q \cong PQ/Q.
\]
Since this minimal $L$-orbit is a Schubert variety, it is contained in the unique Schubert 
divisor of $G/Q$. Let $B_{L}$ denote the Borel subgroup of $L$ as defined at the end of 
the first paragraph in this section. Then $B_{L}$ stabilizes the $B$-orbits (Schubert cells) 
in $G/Q$. In particular, the unique Schubert divisor of $G/Q$ is $B_{L}$-stable. 
It follows from our assumption of toroidalness that the unique Schubert divisor of $G/Q$ is $L$ stable.
We will show that this is not true for the cases under consideration.

Case (a) (See~\cite[Proposition 1.8]{Pasquier}):
We already mentioned that $X^1$ is a quadric hypersurface in $\PP(\C^{m+1}\oplus (\C^{m+1})^*)$,
so $\dim X^1=2m$. 
Here, our group $L$ is isomorphic to $\mathbf{SL}_{m+1}$, and the fundamental representations 
are given by $V(\omega_1) \cong \C^{m+1}$
and $V(\omega_m)\cong (\C^{m+1})^*$. Thus we see that the two (closed) $L$-orbits
are $m$-dimensional projective spaces, $\PP(\C^{m+1})$ and $\PP((\C^{m+1})^*)$. 
Since $m < 2m-1$, no divisor in $X^1$ is $L$-stable.

Case (b) (See~\cite[Proposition 1.9]{Pasquier}):
In this case, we know that $X^1 \cong  \mbf{Gr}(i+1,m+2)$ ($i\in \{1,\dots, m-1\})$, 
so $\dim X^1= (m-i+1)(i+1)$, and we know that 
$V(\omega_i) \cong \bigwedge^i \C^{m+1}$, so 
$\dim G\cdot [v_i] = \dim \mbf{Gr}(i, m+1) = (m-i+1)i$. 
Also, $\dim G\cdot [v_{i+1}] = \dim \mbf{Gr}(i+1, m+1) = (m-i)(i+1)$,
Since $i<m$ and $3\leq m$, we see that neither $G\cdot [v_i]$ or $G\cdot [v_{i+1}]$
is a divisor in $X^1$. In other words, no divisor in $X^1$ is $L$-stable.

Case (c) (See~\cite[Proposition 1.10]{Pasquier}):
In this case, we know that $X^1$ is isomorphic to the 
spinor variety $\mbf{Spin}(2m+1)/P(\omega_m)$.
It is also known as the odd orthogonal Grassmannian,
or the Lagrangian Grassmannian.  
The dimension of $X^1$ is given by $\dim X^1 = {m+1 \choose 2}$, see~\cite[Corollary 6.2.4.4]{LR}.
The two highest weight vectors under consideration, that are $v_{m-1}$ and $v_{m}$,
are actually the two highest weight vectors for the half-spin representations of $\mbf{Spin}(2m)$.
We view $\mbf{Spin}(2m)$ inside $\mbf{Spin}(2m+1)$ as the stabilizer subgroup of a 
suitable basis vector for the standard matrix representation of $\mbf{SO}(2m+1)$, 
which is lifted to $\mbf{Spin}(2m+1)$. In fact, we can choose this particular basis 
vector in such a way that it does not appear in the highest weight vectors $v_{m-1}$ and $v_{m}$. 
Therefore, the stabilizer in $\mbf{Spin}(2m+1)$ of the half-spin representation highest weight vectors
is equal to its intersection with the subgroup $\mbf{Spin}(2m)$ in $\mbf{Spin}(2m+1)$. 
From this we see that 
the dimensions of $\overline{\mbf{Spin}(2m+1)\cdot v_{m-1}}$ and 
$\overline{\mbf{Spin}(2m+1)\cdot v_{m}}$ can be computed by replacing 
$\mbf{Spin}(2m+1)$ with $\mbf{Spin}(2m)$. Then we get two copies of the 
even Grassmannian, which is of dimension ${m\choose 2}$, see~\cite[Corollary 7.2.4.4]{LR}.
Since, for $m\geq 4$, ${m\choose 2} +1 < {m+1 \choose 2}$, we see that 
no divisor in $X^1$ is $L$-stable. This finishes the proof of our theorem.

\end{proof}

\section{Levi Orbits in Schubert Varieties}
\label{S:LeviStable}

In this section we have general results about $L$ orbits in Schubert varieties. 
We begin with fixing a pair $(T,B)$ of maximal torus $T$ and a Borel subgroup 
$B$ containing $T$ in $G$.

Let $Q$ be a standard parabolic subgroup of $G$, 
$X_{wQ}$ ($w\in W^Q$) be a Schubert variety in $G/Q$.
Let $P \subset G$ denote the stabilizer of $X_{wQ}$ in $G$.
Then $P$ is a standard parabolic subgroup as well. 
In this case, any Levi subgroup of $P$ is a maximal
Levi that acts on $X_{wQ}$. 
Let us also assume for convenience that $T\subset L$. 
Set 
$$
B_L:= L\cap B.
$$
Then $B_L$ is a Borel subgroup of $L$.

By Definition~\ref{D:colorandtoroidal}, assuming that 
$X_{wQ}$ is a spherical $L$-variety, 
to see whether $X_{wQ}$ is toroidal or not,
one needs to analyze its $B_L$-stable prime divisors. 

\begin{Remark}\label{R:well known}
Since $B_L\subset B$, any Schubert divisor
is automatically $B_L$-stable. In light of this, we will focus
on describing those prime Schubert divisors that are not $L$-stable, but contain an $L$-orbit.
This will allow us, in Corollary~\ref{C:toroidalGrassNesc}, to provide a set of necessary conditions 
for a Schubert $L$-variety to be toroidal.
\end{Remark}


Towards this end, we first prove a general fact about 
the actions of Levi subgroups
on Schubert varieties. Recall our notation from before; 
unless otherwise stated, 
$Q$ stands for a standard parabolic subgroup in $G$.
First, we have a definition.

\begin{Definition}\label{D:1head}
Let $X_{wQ}$ ($w\in W^Q$) be a Schubert variety in $G/Q$.  
Assume that $X_{wQ}$ is stable under a standard Levi subgroup 
$L$ of a standard parabolic subgroup $P$ of $G$.
An element $\theta\in W^Q$ with $\theta \leq w$ is called 
a {\em degree 1 head} for $w$ if the Schubert subvariety 
$X_{\theta Q}$ is $L$-stable. In particular $w$ is a degree 1 head
for itself. 
\end{Definition}

\begin{Proposition}
\label{P:LorbitCriterion}
We preserve the notation from Definition~\ref{D:1head}.
Let $\tau \in W^Q$ with $\tau \leq w$. 
Then the Schubert variety $X_{\tau}$ contains an $L$-orbit 
if and only if there exists a degree 1 head $\theta$ such that $\theta \leq \tau$.
\end{Proposition}
\begin{proof}
One direction is immediate, if $X_{\tau}$ contains an $L$-stable Schubert subvariety then clearly it must contain an $L$-orbit. For the other direction, let us suppose that $X_{\tau}$ contains an $L$-orbit. Let $x$ be a point in this orbit. Then we claim that the $P$-orbit through the point $x$ is contained in $X_{\tau}$. Any element in the orbit may be written as $\ell u x$ for $\ell \in L$ and $u \in U_P$. As $P$ is a semidirect product of $L$ and $U_P$ we know that there exists an $\ell' \in L$ and $u' \in U_P$ such that $\ell u x = u' \ell' x$. Since $l' x \in X_{\tau}$ and $U_P \subset B$ it follows that $\ell u x = u' \ell' x \in X_{\tau}$. Hence both $P \cdot x$ and $\overline{P \cdot x}$ are contained in $X_{\tau}$. But $\overline{P \cdot x}$ is a $B$-stable subvariety of $X_{\tau}$ which implies, via the Bruhat decomposition and the fact that $P$ is connected, that $\overline{P \cdot x} = X_{\gamma}$ for some $\gamma \in W^Q$. The left hand side of this equality is $P$-stable and contained in $X_{\tau}$, and hence so is $X_{\gamma}$. This concludes our proof, since this gives that $X_{\gamma}$ is an $L$-stable Schubert subvariety of $X_{\tau}$ (note it is possible that $\gamma = \tau$); equivalently $\gamma$ is a degree 1 head less than or equal $\tau$.
\end{proof}

\begin{Corollary}
\label{C:Lhomogenous}
We preserve the notation from Proposition~\ref{P:LorbitCriterion}.
An $L$-stable Schubert variety $X_{w Q}$ in $G / Q$ is either 
$L$-homogeneous or contains an $L$-stable Schubert subvariety $X_{\tau Q}$ with $X_{\tau Q}\neq X_{w Q}$
\end{Corollary}

\begin{proof}
We begin by proving that $X_{w Q}$ is $L$-homogeneous if and only if it is $P$-homogeneous. One direction is obvious, for the other we assume that $X_{w Q}$ is $P$-homogeneous. In this case, there is a single $P$-orbit and this orbit must contain the point $id Q$. Let $K$ be the stabilizer subgroup of $P$ at this point, then $X_{w Q}\cong P / K$. Clearly $K$ must contain $B$, as $id Q$ is a $B$-fixed point. But this implies that the inclusion map $L \hookrightarrow P$ descends to a surjection $L \twoheadrightarrow P / K$. Hence there exists a subgroup $J$ in $L$ such that $L / J \cong P / K \cong X_{w Q}$,
where $J=K\cap L$. 

Now suppose that we have an $L$-stable Schubert variety $X_{w Q}$. If it is $L$-homogeneous we are done. If it is not $L$-homogeneous, then by the above argument we know it is not $P$-homogeneous. Thus there are at least two $P$-orbits, and there must exist a $P$-orbit that is not dense. The closure of this non-dense $P$-orbit is $B$-stable and thus is a Schubert variety $X_{\tau Q}$. The fact that the $P$-orbit was not dense implies that $X_{\tau Q} \neq X_{w Q}$. Further, this Schubert variety is $P$-stable and hence $L$-stable.
\end{proof}

\begin{Remark}\label{R:uniqueminimal}
Our previous result implies that the minimal $L$-stable Schubert varieties are all $L$-homogeneous. 
In fact, more is true; for a fixed $L$ there is a unique minimal $L$-stable, $L$-homogeneous Schubert variety. 
It is precisely the $L$-orbit through $id Q$ and is isomorphic to $L / L \cap Q  \cong P / P \cap Q$.
\end{Remark}

\section{Schubert Divisors of Grassmann Schubert Varieties}
\label{S:ToroidalinGrassmann}

In this section, we specialize to type A flag varieties. 
Let $\mbf{T}_n$ and $\mbf{B}_n$ denote, respectively, 
the maximal diagonal torus and the Borel subgroup
consisting of $n\times n$ upper triangular matrices in $\mbf{GL}_n$. 
Let $(\mbf{\Phi},\mbf{\Delta})$ denote the root system and 
the set of simple roots determined by the data of 
$(\mbf{GL}_n,\mbf{B}_n,\mbf{T}_n)$.
The Weyl group of $\mbf{GL}_n$ is the group of 
permutations of $\{1,\dots, n\}$, and we denote it by $\mbf{S}_n$. 
Suppose that 
$$
\mbf{\Delta}:=\{\alpha_1,\dots, \alpha_{n-1}\}.
$$
is the set of simple roots with the standard (Bourbaki) ordering.
Let $\mbf{Q}_d$ denote the maximal (standard) 
parabolic subgroup corresponding to $\mbf{\Delta} \setminus \{\alpha_d\}$
for some simple root $\alpha_d$ from $\mbf{\Delta}$.
In this case, we denote the Weyl group of $\mbf{Q}_d$ by 
$\mbf{S}_{n,d}$; it is isomorphic to $\mbf{S}_d\times \mbf{S}_{n-d}$. 
The set of minimal length left coset representatives 
of $\mbf{S}_{n,d}$ in $\mbf{S}_n$,
which we denote by $\mbf{S}_n^d$, consists of permutations 
$w= (w_1\dots w_n)$ such that 
$$
w_1 <\dots < w_d, \ w_{d+1}< \dots < w_n.
$$

Note that in all cases except for the identity we have $w_d > w_{d+1}$. For the rest of this section we will assume that $w$ is not the identity, as this will make our results simpler to state, and for the identity the questions we are interested in become trivial. When $w$ is the identity, the associated Schubert variety is a point space and thus is a spherical (resp. toroidal) $L$-variety if and only if it is $L$-stable.

Under the assumptions of the previous paragraph, 
we have the precise combinatorial 
description of degree 1 heads from \cite{HodgesLakshmibai}.
Since it is useful for our purposes, we briefly review this development.

Let $I=\{\alpha_{i_1},\dots, \alpha_{i_s}\}$ be a subset of $\mbf{\Delta}$ 
with $\Comp{I}=\{\alpha_{j_1},\dots, \alpha_{j_r}\}$ its complement in $\mbf{\Delta}$. 
For the simplicity of notation, we 
identify $I$ with the set of indices $\{ i_1,\dots, i_s\}$ and $\Comp{I}$ with $\{ j_1,\dots, j_r\}$
and we apply this convention to all subsets of $\mbf{\Delta}$.
Let $\mbf{L}$ denote the standard Levi factor of the
(standard) parabolic subgroup $\mbf{P}$ whose simple roots are $I$. 
Also determined by $I$ and $\Comp{I}$ is a set partition of $\{1,\dots, n\}$,
whose subsets are denoted by 
$\mathrm{Block}_{\mbf{L},1},\ldots,\mathrm{Block}_{\mbf{L},r+1}$, 
where the maximal element of each $\mathrm{Block}_{\mbf{L},k}$ 
is $j_k$ (or $n$ for $\mathrm{Block}_{\mbf{L},r+1}$). 
For example, for $I = \{ 1,3,4,7 \}$, $\Comp{I} = \{ 2,5,6 \}$, and $n=8$ we have 
$\mbf{L} = \mbf{GL}_2 \times \mbf{GL}_3 \times \mbf{GL}_1 \times \mbf{GL}_2 \subset \mbf{GL}_8$. 
Then $\mathrm{Block}_{\mbf{L}, 1}= \{1,2\},\mathrm{Block}_{\mbf{L}, 2} = \{3,4,5\}$, 
$\mathrm{Block}_{\mbf{L}, 3}=\{6\}$, and $\mathrm{Block}_{\mbf{L}, 4}=\{7, 8\}$. 
We define $b_{\mbf{L}}$ to be the \emph{number of blocks}, that is $|\Comp{I}| + 1$.
\\

{\em Notation: The one-line notation for a permutation $\tau$ 
in $\mbf{S}_n$ is the parenthesized list of numbers $(\tau_1\dots \tau_n)$, where 
$\tau_i$ ($1\leq i \leq n$) denotes the value of $\tau$ at $i$ viewed as a function on $\{1,\dots, n\}$.  
If $\tau$ is a permutation with one-line notation $\tau= (\tau_1\dots \tau_n)$, 
then we denote by $\tau^{(d)}$ the sequence $(\tau_1,\dots,\tau_d) \uparrow$, where $\uparrow$ indicates that 
the sequence has been reordered so that it is increasing.}
\\

Let $w, \theta \in \mbf{S}_n^d$ be two permutations 
such that the Schubert variety 
$X_{w\mbf{Q}_d}$ is $\mbf{L}$-stable and $\theta$ is a degree 1 head. 
In~\cite{HodgesLakshmibai}, the requirement that 
$\theta$ is a degree 1 head is shown 
to be equivalent to the following combinatorial criterion. 

\begin{Proposition}\label{P:maxConditionHeads}
Let $\theta \in W^Q$ with $\theta \leq w$. Then $\theta$ is a degree 
1 head if and only if $\theta^{(d)} \cap \mathrm{Block}_{\mbf{L},k}$ 
is a maximal collection of elements in $\mathrm{Block}_{\mbf{L},k}$ 
for all $1 \leq k \leq b_{\mbf{L}}$.
\end{Proposition}

Note that if the Schubert variety $X_{w\mbf{Q}_d}$ is smooth, 
then $w=(w_1\dots w_n)$ has the form 
$$
w= (1\dots  p\,\,  m\,\, m+1 \dots  m+d-p+1\,\, w_{d+1}\dots w_n),
$$
where $p \geq 0$ and $m > p + 1$;  
if $p$ is 0, then the initial entry of the permutation starts with $m$.
In this notation, the maximal standard Levi subgroup that acts on
$X_{w\mbf{Q}_d}$ is given by $\mbf{L}_{\mt{max}}:=\mbf{L}_{I_w}$, 
where $\Comp{I_w}$ contains $p$ if it is greater 
than $0$ and $m+d-p+1$ if it is less than $n$. 
In particular, this implies, in the 
most general case, that $\mathrm{Block}_{\mbf{L}_{\mt{max}},1} = \{ 1,\ldots,p\}$, 
$\mathrm{Block}_{\mbf{L}_{\mt{max}},2} = \{ p+1,\ldots,m+d-p+1\}$, 
and $\mathrm{Block}_{\mbf{L}_{\mt{max}},3} = \{ m+d-p+2 , \ldots , n\}$.

\begin{Theorem}\label{T:toroidal in a grassmannian}
Let $X_{w\mbf{Q}_d}$ be a smooth Schubert variety
and let $\mbf{L}_{\mt{max}}$ be the (maximal) Levi 
as defined in the previous paragraph. Then 
the Schubert divisors in $X_{w\mbf{Q}_d}$ 
do not contain any $\mbf{L}_{\mt{max}}$-orbits.
\end{Theorem}
\begin{proof}
It is not difficult to argue in this case that 
the only degree 1 head is $w$ itself. 
To see this, first, observe that any element $\tau \leq w$ is of the form
\begin{center}
$\tau=(1\dots p\,\, j_1 \dots j_{d-p}\,\, w_{d+1}\dots w_n)$
\end{center}
with each $j_b \leq m+d-p+1$. 
This means that the number of elements in 
$\tau^{(d)} \cap \mathrm{Block}_{\mbf{L}_{\mt{max}},1}$ 
is always $p$, the number of elements in 
$\tau^{(d)} \cap \mathrm{Block}_{\mbf{L}_{\mt{max}},2}$ 
is always $d-p$, and the number of elements in 
$\mathrm{Block}_{\mbf{L}_{\mt{max}},3}$ is always 0. 
As a degree 1 head must be a maximal collection 
of elements in each of these sets, the head is 
uniquely determined by the number of entries 
in each block, and thus $w$ is in fact the unique 
degree 1 head.

Now, the uniqueness property implies that for any Schubert 
divisor $X_{\tau}$ there does not exist a degree 
1 head less than or equal to $\tau$. 
Using Proposition~\ref{P:LorbitCriterion} 
we conclude that there are no $\mbf{L}_{\mt{max}}$-orbits 
contained in $X_{\tau\mbf{Q}_d}$.
\end{proof}

\begin{Corollary}\label{C:wonderfulgrassmannian}
Let $X_{w\mbf{Q}_d}$ be a smooth Schubert variety in 
$\mbf{GL}_n/\mbf{Q}_d$. If $\mbf{L}_{\mt{max}}$ denotes the 
standard Levi factor of the stabilizer $\mbf{P}_{\mt{max}}:=\mt{Stab}_{\mbf{GL}_n}(X_{w\mbf{Q}_d})$,
then $X_{w\mbf{Q}_d}\cong \mbf{L}_{\mt{max}}/\mbf{P}$,
where $\mbf{P}$ is a maximal parabolic 
subgroup of $\mbf{L}_{\mt{max}}$. 
\end{Corollary}
\begin{proof}
It follows from 
Theorem~\ref{T:toroidal in a grassmannian} that $\mbf{L}_{\mt{max}}$ 
acts transitively on $X_{w\mbf{Q}_d}$,
hence, $X_{w\mbf{Q}_d}$ is a homogenous variety for $\mbf{L}_{\mt{max}}$. 
Since smooth Schubert varieties 
in Grassmannians are Grassmann varieties, 
the stabilizer of a point $p\in X_{w\mbf{Q}}$ 
in $\mbf{L}_{\mt{max}}$ is a maximal parabolic subgroup. 
This proves the second claim. The first claim follows from the second. 
\end{proof}

Now we assume that $w$ is an arbitrary {\em Grassmann permutation} 
from $\mbf{S}_n^d$. 
In other words, $w$ has a single descent at the $d$-th position. 
Then, in one-line notation, $w$ has the form 
\begin{align}\label{A:grassmannperm}
w= (a_1\,\, a_1+1 \dots a_1+b_1\,\, a_2\,\, (a_2+1) \dots a_2+b_2\,\, 
\dots a_j\,\, (a_j+1) \dots a_j+b_j\,\, w_{d+1} \dots w_n),
\end{align}
where $a_\ell + b_\ell + 1 < a_{\ell+1}$ for $1 \leq \ell < j$.
If we write $w$ more briefly as in $w=(w_1\,\dots \,w_n)$, then 
the dimension of the Schubert variety $X_{w\mbf{Q}_d}$ is given by  
\begin{align*}
\dim X_{w \mbf{Q}_d} &= \sum_{i=1}^d (w_i-i),
\end{align*}
see~\cite[Section 1.1]{BrionLectures}.
Let $w_\ell$ denote $w_\ell := s_{\alpha_{a_\ell}} w$,
so that, in the one-line notation of (\ref{A:grassmannperm}), 
$w_\ell$ is obtained from $w$ by replacing a single $a_\ell$ by 
$a_\ell - 1$ for some $1 \leq \ell \leq j$ such that $a_\ell > 1$. 
Recall our convention that we identify the subsets of 
the set of simple roots $\mbf{\Delta}$ by the sets of indices
of the roots that they contain.

\begin{Theorem}\label{T:divisorsColors}
Let $\mbf{L}:= \mbf{L}_I$ ($I\subset \mbf{\Delta}$) 
be a standard Levi subgroup in $\mbf{GL}_n$ and let 
$w \in \mbf{S}_n^d$ be a (arbitrary) Grassmann permutation as in (\ref{A:grassmannperm}) 
such that the associated Schubert variety $X_{w\mbf{Q}_d}$ is $\mbf{L}$-stable.
Then any Schubert divisor of $X_{w\mbf{Q}_d}$ is of the form $X_{w_\ell \mbf{Q}_d}$ for some $1 \leq \ell \leq j$.
Furthermore, the Schubert divisor $X_{w_\ell \mbf{Q}_d}$ is $\mbf{L}$-stable 
(equivalently $w_\ell$ is a degree 1 head) if and only if $\Comp{I}$ contains $a_\ell - 1$.
\end{Theorem}

\begin{proof}

The fact that any Schubert divisor is of the form $X_{w_\ell\mbf{Q}_d}$ 
is a standard result; it follows, for example, from the standard combinatorial 
description of the Bruhat order on Grassmann permutations in addition to the 
dimension formula for Schubert varieties in the Grassmannian. 
By Proposition~\ref{P:maxConditionHeads}, we know that a Schubert variety 
$X_{w_\ell\mbf{Q}_d}$ will be a degree 1 head if and only if 
$w_\ell^{(d)} \cap \mathrm{Block}_{\mbf{L},k}$ is a maximal 
collection of elements in $\mathrm{Block}_{\mbf{L},k}$ for all 
$1 \leq k \leq b_{\mbf{L}}$. Note that $w_\ell$ contains the value 
$a_\ell - 1$ but does not contain $a_\ell$. It immediately follows that if 
$X_{w_\ell \mbf{Q}_d}$ is $\mbf{L}$-stable then $a_\ell - 1$ must be the 
maximal element in some block if the maximality condition 
from Proposition~\ref{P:maxConditionHeads} is to hold. This is equivalent to $\Comp{I}$ 
containing $a_\ell - 1$. Now we suppose that $I$ contains $a_\ell - 1$. 
The fact that $\Comp{I}$ contains $a_k+b_k$ follows from the fact that $X_{w\mbf{Q}_d}$ is 
$\mbf{L}$-stable ($w$ is a degree 1 head) and the Proposition~\ref{P:maxConditionHeads} maximality 
condition. Thus it is not hard to see that any intersection 
$w_\ell^{(d)} \cap \mathrm{Block}_{\mbf{L},k}$ is a maximal collection of elements in 
$\mathrm{Block}_{\mbf{L},k}$
\end{proof}

\begin{Corollary}\label{C:GrassmannToroidal}
\label{C:toroidalGrassNesc} We preserve the notation from Theorem~\ref{T:divisorsColors}.
If the Schubert variety $X_{w\mbf{Q}_d}$ is toroidal, then one of the following two criteria is satisfied for
all Schubert divisors $X_{w_\ell\mbf{Q}_d}$:
\begin{enumerate}
\item $a_\ell - 1 \in \Comp{I}$, or 
\item $a_\ell - 1 \notin \Comp{I}$ and there exist no $L_I$-stable Schubert subvarieties in $X_{w_\ell\mbf{Q}_d}$.
\end{enumerate}
\end{Corollary}
\begin{proof}
By Theorem~\ref{T:divisorsColors}, if the first criteria holds then $X_{w_\ell\mbf{Q}_d}$ is $L_I$-stable. 
If the second criteria holds, $a_\ell - 1 \notin \Comp{I}$ implies that $X_{w_\ell\mbf{Q}_d}$ is not $L_I$-stable.
Hence, by Proposition~\ref{P:LorbitCriterion}, it will contain an $L_I$-orbit if and only if it strictly contains an $L_I$-stable Schubert subvariety.
\end{proof}

\section{The Singular Locus of a Schubert Variety} 
\label{S:Singularlocus}

In this section, we will focus on the singular locus of a Schubert variety in $G/Q$, where $G$ is a 
simple algebraic group. We do not
make any sphericity assumptions on the actions of the 
subgroups of $G$.
\vspace{.25cm}

First, we assume that $G$ is arbitrary 
(but as in Section~\ref{S:Preliminaries}).
Let $Q$ be a standard parabolic subgroup in $G$.
Let $X_{wQ}$ ($w\in W^Q$) be a 
singular Schubert variety and let $L_{max}$ 
denote the standard Levi subgroup of 
$P:=\mt{Stab}_G(X_{wQ})$.  
We denote by $\mt{Sing}(X_{wQ})$
the singular locus of $X_{wQ}$, which is a 
union of Schubert varieties. Then by normality
(Serre's criterion) we have 
$$
\dim \mt{Sing}(X_{wQ}) + 2\leq \dim X_{wQ}.
$$
Let $x$ be a point from $\mt{Sing}(X_{wQ})$. 
Since $L_{max} \cdot x$ consists of singular points, 
we see that $\mt{Sing}(X_{wQ})$ is 
a scheme theoretic union
of certain $L$-stable Schubert subvarieties $X_{\tau_1Q},\dots, X_{\tau_k Q}$,
where $\tau_i \in W^Q$ for $i=1,\dots, k$. 

Next, we restrict our attention to the case where $G$
is a simple algebraic group. 
The irreducible representations of $G$ are parametrized by the semigroup 
of dominant weights with respect to $T$, and furthermore, the weight 
lattice of $(G,T)$ is freely generated by the (fundamental) dominant weights corresponding
to the fundamental representations of $G$.
A dominant weight $\lambda$ is called {\em minuscule} 
if $\langle \lambda , \check{\alpha} \rangle  \leq 1$ 
for all positive coroots $\check{\alpha}$. 
A fundamental weight $\lambda$ is called {\em cominuscule} if 
the associated simple root occurs in the highest root with coefficient one. 
In our next result, we talk about minuscule and cominuscule
parabolic subgroups $Q$ in $G$. By this, in the former case, we mean 
a maximal parabolic subgroup $Q$ associated with a minuscule 
fundamental weight, and in the latter case we mean a maximal 
parabolic subgroup $Q$ associated with a cominuscule fundamental 
weight.

\begin{Proposition}\label{P:singularlocus}
Let $Q$ be a standard, minuscule or cominuscule parabolic subgroup in $G$.
Let $X_{wQ}$ ($w\in W^Q$) be a singular Schubert $L_{max}$-variety, where
$L_{max}$ is the standard Levi subgroup of the stabilizer subgroup $P:=\mt{Stab}_G(X_{wQ})$.
In this case, all $L_{max}$-stable Schubert subvarieties of $X_{wQ}$ are contained in 
the singular locus $\mt{Sing}(X_{wQ})$. 
\end{Proposition}

\begin{proof}
By Proposition~\ref{P:LorbitCriterion}, 
we know that $X_{\tau Q}$ is $P$-stable. 
Let $e_w$ denote the unique $T$-fixed point 
in the dense $B$-orbit of $X_{wQ}$. 
Clearly, $e_w \notin X_{\tau Q}$.
Otherwise, $X_{\tau Q}$ would be equal to 
$X_{wQ}$. On the other hand, 
by a famous result of 
Brion and Polo~\cite{BrionPolo} 
we know that 
the singular locus 
of a Schubert variety in a (co)minuscule 
partial flag variety is precisely the boundary 
of the Schubet variety. Here, 
the boundary of $X_{wQ}$ is
the complement
$$
\mt{Bd}(X_{wQ}) := X_{wQ} \setminus P \cdot e_w.
$$
Since $e_w \notin X_{\tau Q}$ and 
$X_{\tau Q}$ is $P$-stable, we see 
that $P\cdot e_w \cap X_{\tau Q} = \emptyset$.
Therefore, $X_{\tau Q} \subseteq \mt{Bd}(X_{wQ})$.
This finishes the proof. 

\end{proof}

\begin{Remark}
Let $X_{wQ}$ be an $L$-stable Schubert variety in $G/Q$,
where $Q$ is a parabolic subgroup of $G$.
As we mentioned in the second paragraph of this section,
the irreducible components of $\mt{Sing}(X_{wQ})$ 
are $L$-stable Schubert subvarieties of $X_{wQ}$. 
Note that this fact holds true for any 
connected subgroup $L$ of $P:=\mt{Stab}_G(X_{wQ})$.
Proposition~\ref{P:singularlocus} implies that, if $Q$
is a standard (co)minuscule parabolic subgroup and $X_{wQ}$ 
is an $L_{max}$-variety where $L_{max}$ is the standard Levi 
subgroup of $P$, then the
irreducible components of $\mt{Sing}(X_{wQ})$ are 
maximal $L_{max}$-stable Schubert subvarieties of $X_{wQ}$. 
Unfortunately, this maximality property does not hold 
true in general. For example, consider $\mbf{GL}_4/\mbf{B}_4$
and the Schubert subvariety $X_{(3412)}$. 
This Schubert variety is singular and its singular locus 
is precisely $X_{(1324)}$. The standard Levi subgroup 
of the stabilizer of $X_{(3412)}$ in $\mbf{GL}_4$ is 
$\mbf{L}_{\mt{max}}:=\mbf{L}_{\{\alpha_2 \}}$ and while $X_{(1324)}$ is $\mbf{L}_{\mt{max}}$-stable, 
the maximal $\mbf{L}_{\mt{max}}$-stable Schubert subvarieties of $X_{(3412)}$ are 
the divisors $X_{(1432)}$, $X_{(3142)}$, and $X_{(3214)}$. 
It would be of great interest to have an algebraic group-theoretic 
criterion for those $L$-stable Schubert subvarieties that comprise 
the singular locus in a general $G / Q$.
\end{Remark}

\begin{Corollary}\label{C:noLstable}
We preserve the hypothesis from Proposition~\ref{P:singularlocus}. 
In that case, $X_{wQ}$ has no $L_{max}$-stable Schubert divisors.  
In particular, $X_{wQ}$ is not a toroidal $L_{max}$-variety. 
\end{Corollary}

\begin{proof}
It is well known that Schubert varieties are normal 
(see, for example,~\cite{Seshadri87}).
Therefore, $\mt{Sing}(X_{wQ})$ has
codimension at least 2 in $X_{wQ}$. By Proposition~\ref{P:singularlocus},
any $L_{max}$-stable Schubert variety $X_{\tau Q}$ in $X_{wQ}$ is contained
in $\mt{Sing}(X_{wQ})$. This shows that there are no $L_{max}$-stable
Schubert divisors in $X_{wQ}$. 

To prove the last claim, we recall that there is always a minimal $L_{max}$-homogeneous 
Schubert subvariety. Therefore, one of the Schubert divisors of $X_{wQ}$ contains an $L_{max}$-orbit,
but this Schubert divisor is not $L_{max}$-stable. This finishes the proof of our claim. 
\end{proof}

\section{Toroidalness Under BP Decomposition}
\label{S:BPdecomposition}

In this section we analyze how the smooth toroidal Schubert 
varieties change under the canonical projections 
$G/B\rightarrow G/Q$. Our main result is stated 
for $G=\mbf{GL}_n$. 
Once again, we start with mentioning some general results.

Let $P$ and $Q$ be two parabolic subgroups in $G$.
Assume that $P\subset Q$ and let $w\in W^P$. 
The {\em parabolic decomposition} of $w$ with respect
to $Q$ is the unique decomposition $w= vu$, where 
$v\in W^Q$, $u\in W_Q\cap W^P$, hence, the 
generic fiber of the projection 
$$\pi:=p_{P,Q} \vert_{X_{wP}} :\ X_{wP}\rightarrow X_{vQ}$$ 
is isomorphic to $X_{uP}$.

A parabolic decomposition $w=vu$ of $w$ is 
said to be a BP-decomposition if the Poincar\'e 
polynomial of $X_{wP}$ is the product of the
Poincar\'e polynomials of $X_{uP}$ and 
$X_{vQ}$. There are various equivalent characterizations
of BP-decompositions.
We state two of them; for details, see~\cite[Proposition 4.2]{RichmondSlofstra}.
\begin{enumerate}
\item $w=vu$ is a BP-decomposition if and only if 
$u$ is a maximal element of the intersection 
\begin{align}\label{A:u is maximal}
[id, w] \cap W^P\cap W_Q.
\end{align}
\item 
The {\em support} of an element $v\in W$, denoted
by $S(v)$, is the set 
$$
S(v) := \{ s \in S:\ s \leq v \}.
$$
By $D_L(u)$ we denote the 
{\em left descent set} of $u$,
$$
D_L(u) := \{ s\in S:\ \ell(su) < \ell(u) \}.
$$
Let $J$ and $K$ denote two subsets of the 
Coxeter generators $S$ of $W$ such that 
$W_P = W_I$ and $W_Q=W_K$.
In this notation, the parabolic decomposition $w=vu$ 
is a BP-decomposition
of $w$ if and only if 
$$
S(v) \cap K \subseteq D_L(u'),
$$
where $u'$ is the maximal element of the
coset $uW_P$. 
\end{enumerate}

\begin{Remark}
Also from~\cite{RichmondSlofstra}, we know that  
$X_{wP}$ is (rationally) smooth if and only if 
there exists a maximal parabolic subgroup $Q$ containing $P$
and a BP-decomposition $w= vu$ with $v\in W^Q$ and $u\in W^P \cap W_Q$ 
such that $X_{vQ}$ and $X_{uP}$ are (rationally) smooth Schubert varieties,
\end{Remark}

As we alluded at the beginning of this section, 
our goal here is to understand the change in 
the divisors of $X_{wP}$ under a BP-decomposition. 
We proceed with the assumption that $X_{wP}$ 
is a smooth Schubert variety. 
It is well known that for such Schubert 
varieties, the Picard group $\Pic {X_{wP}}$ of line bundles
on $X_{wP}$ is a finitely 
generated free abelian group and it is isomorphic to the divisor 
class group $\mt{Cl}(X_{wP})$.
At the same time, as a consequence of the fiber bundle property
of the BP-decomposition of a smooth Schubert variety, 
we know that the projection $\pi: X_{wP} \rightarrow X_{vQ}$ has a section,
therefore, the induced map 
$$
\gamma^* :\mt{Cl}(X_{wP}) \rightarrow \mt{Cl}(X_{vQ})
$$
is a surjection. 
Moreover, all of these fit to a commutative diagram as follows:
\begin{figure}[h]
\begin{center}
\begin{tikzpicture}
\node at (-4.25,1) (a) {$\Pic {G/P} \cong \mt{Cl}(G/P)$};
\node at (4.25,1) (b) {$\Pic {G/Q} \cong \mt{Cl}(G/Q)$};
\node at (-4.25,-1) (c) {$\Pic {X_{wP}} \cong \mt{Cl}(X_{wP})$};
\node at (4.25,-1) (d) {$\Pic {X_{vQ}} \cong \mt{Cl}(X_{vQ})$}; 
\node at (-2.7,0) {}; 
\node at (2.6,0) {}; 
\node at (0,1.35) {$p_{P,Q}^*$}; 
\node at (0,-1.35) {$\gamma^*$}; 
\draw[->>,thick] (a) to (b);
\draw[->>,thick] (a) to (c);
\draw[->>,thick] (b) to (d);
\draw[->,thick] (c) to (d);
\end{tikzpicture}
\end{center}
\end{figure}

\begin{Remark}\label{R:ST}
Let $X_{\tau P}$ be a Schubert divisor in $X_{w P}$.
Then $\tau = s w=s vu$ for some $s\in S(w)$. 
We have two remarks in order:
\begin{enumerate}
\item As a consequence of the above observation on the class groups, 
we see that either $\pi (X_{\tau P}) = \pi (X_{w P})$,
or $\pi (X_{\tau P})$ is the unique Schubert divisor in $X_{v Q}$. 
If $s\notin W_Q$, then $p_{P,Q} (X_{wQ}) \neq p_{P,Q}(X_{\tau Q})$.

\item Since $X_{wP}$ is $L$-stable, and since $\pi$ is 
the restriction of $p_{P,Q}$, 
which is $G$-equivariant, we see that $X_{vQ}$ is $L$-stable 
as well. In particular, the fibers of $\pi$ are $L$-stable. 
\end{enumerate}
\end{Remark}

Now we specialize to the type A situation. 
Let $\mathbf{L}$ be a standard Levi subgroup 
and $\mathbf{P}$ be a standard parabolic subgroup in $\mathbf{GL}_n$.

\begin{Theorem} \label{T:divisors to divisors1}
Let $X_{w\mathbf{P}}$ be a smooth spherical Schubert 
$\mathbf{L}$-variety in $\mathbf{GL}_n/\mathbf{P}$, and 
let $w=vu$ be a BP decomposition with the corresponding projection 
$\pi : X_{w \mathbf{P}} \rightarrow X_{v\mathbf{Q}}$, 
where $\mathbf{Q}$ is a maximal parabolic subgroup of 
$\mathbf{GL}_n$. 
If $X_{w\mathbf{P}}$ is a smooth toroidal Schubert
$\mathbf{L}$-variety, then so are $X_{v\mathbf{Q}}$ and $X_{u\mathbf{P}}$. 
\end{Theorem}

\begin{proof}
The smoothness part of our statement follows from the above discussion,
so, we will prove toroidalness. 
First, we will show that $X_{v\mathbf{Q}}$ is a 
toroidal Schubert $\mathbf{L}$-variety.

Note that $\pi$ is an $\mbf{L}$-equivariant, faithfully flat, proper morphism
of varieties. Therefore, it pull backs divisors to divisors, see~\cite[Chapter I.1.7]{Fulton}. 
Assume that $X_{v\mbf{Q}}$ is not toroidal. Then there exists a 
$\mbf{B}_\mbf{L}$-stable but not $\mbf{L}$-stable divisor $D$ in $X_{v\mbf{Q}}$ such 
that $D$ contains an $\mbf{L}$-orbit. 
But then the same is true for the pull-back divisor $\pi^* D$ in $X_{w\mbf{P}}$.
This contradicts the toroidalness of $X_{w\mbf{P}}$. 
Therefore, $X_{v\mbf{Q}}$ is a toroidal $\mbf{L}$-variety.

Next, we will show that $X_{u\mbf{P}}$ is toroidal. 
We already know that it is smooth, and we know that 
it is $\mbf{L}$-stable by Remark~\ref{R:ST}. 
Since $X_{w\mbf{P}}$ is spherical, so is $X_{u\mbf{P}}$.
But more is true; $X_{w\mbf{P}}$ is toroidal, hence, 
it is a regular $\mbf{L}$-variety. Thus, any $\mbf{L}$-orbit closure 
in $X_{w\mbf{P}}$ is a regular variety.
In particular, $X_{u\mbf{P}}$ is regular, hence, it is a toroidal $\mbf{L}$-variety
by Lemma~\ref{L:BB}.

\end{proof}

Combined with Corollary~\ref{C:GrassmannToroidal}, Theorem~\ref{T:divisors to divisors1} gives 
a nontrivial necessary condition for a Schubert variety to be a toroidal Schubert variety. 

\bibliography{References}

\begin{thebibliography}{10}

\bibitem{AvdeevPetukhov}
Roman~S. Avdeev and Alexey~V. Petukhov.
\newblock Spherical actions on flag varieties.
\newblock {\em Mat. Sb.}, 205(9):4--48, 2014.

\bibitem{AvdeevPetukhov2}
Roman~S. {Avdeev} and Alexey~V. {Petukhov}.
\newblock {Spherical actions on isotropic flag varieties and related branching
  rules}.
\newblock {\em arXiv e-prints}, page arXiv:1812.00936, December 2018.

\bibitem{BienBrion}
Fr\'ed\'eric Bien and Michel Brion.
\newblock Automorphisms and local rigidity of regular varieties.
\newblock {\em Compositio Math.}, 104(1):1--26, 1996.

\bibitem{BDP90}
Emili Bifet, Corrado De~Concini, and Claudio Procesi.
\newblock Cohomology of regular embeddings.
\newblock {\em Adv. Math.}, 82(1):1--34, 1990.

\bibitem{BrionSpherique}
Michel Brion.
\newblock Vari\'eti\'es {S}ph\'erique.
\newblock Notes de la session de la S. M. F. ``Op{\'e}rations hamiltoniennes et
  op{\'e}rations de groupes alg{\'e}briques'',
  \url{https://www-fourier.ujf-grenoble.fr/~mbrion/spheriques.pdf}, 1997.

\bibitem{Brion98}
Michel Brion.
\newblock The behaviour at infinity of the {B}ruhat decomposition.
\newblock {\em Comment. Math. Helv.}, 73(1):137--174, 1998.

\bibitem{Brion01}
Michel Brion.
\newblock On orbit closures of spherical subgroups in flag varieties.
\newblock {\em Comment. Math. Helv.}, 76(2):263--299, 2001.

\bibitem{BrionLectures}
Michel Brion.
\newblock Lectures on the geometry of flag varieties.
\newblock In {\em Topics in cohomological studies of algebraic varieties},
  Trends Math., pages 33--85. Birkh\"auser, Basel, 2005.

\bibitem{BrionPolo}
Michel Brion and Patrick Polo.
\newblock Generic singularities of certain {S}chubert varieties.
\newblock {\em Math. Z.}, 231(2):301--324, 1999.

\bibitem{Caldero}
Philippe Caldero.
\newblock Toric degenerations of {S}chubert varieties.
\newblock {\em Transform. Groups}, 7(1):51--60, 2002.

\bibitem{CanHodges}
Mahir~Bilen Can and Reuven Hodges.
\newblock Smooth schubert varieties are spherical.
\newblock \url{https://arxiv.org/abs/1803.05515}, 2018.

\bibitem{Fulton}
William Fulton.
\newblock {\em Intersection theory}, volume~2 of {\em Ergebnisse der Mathematik
  und ihrer Grenzgebiete (3) [Results in Mathematics and Related Areas (3)]}.
\newblock Springer-Verlag, Berlin, 1984.

\bibitem{Ginsburg}
Victor Ginsburg.
\newblock Admissible modules on a symmetric space.
\newblock {\em Ast\'erisque}, (173-174):9--10, 199--255, 1989.
\newblock Orbites unipotentes et repr\'esentations, III.

\bibitem{HodgesLakshmibai}
Reuven Hodges and Venkatramani. Lakshmibai.
\newblock Levi subgroup actions on schubert varieties, induced decompositions
  of their coordinate rings, and sphericity consequences.
\newblock {\em Algebras and Representation Theory}, 21(6):1219--1249, Dec 2018.

\bibitem{LR}
Venkatramani Lakshmibai and Komaranapuram~N. Raghavan.
\newblock {\em Standard monomial theory}, volume 137 of {\em Encyclopaedia of
  Mathematical Sciences}.
\newblock Springer-Verlag, Berlin, 2008.
\newblock Invariant theoretic approach, Invariant Theory and Algebraic
  Transformation Groups, 8.

\bibitem{Littelmann}
Peter Littelmann.
\newblock On spherical double cones.
\newblock {\em J. Algebra}, 166(1):142--157, 1994.

\bibitem{MWZ1}
Peter Magyar, Jerzy Weyman, and Andrei Zelevinsky.
\newblock Multiple flag varieties of finite type.
\newblock {\em Adv. Math.}, 141(1):97--118, 1999.

\bibitem{MWZ2}
Peter Magyar, Jerzy Weyman, and Andrei Zelevinsky.
\newblock Symplectic multiple flag varieties of finite type.
\newblock {\em J. Algebra}, 230(1):245--265, 2000.

\bibitem{Pasquier}
Boris Pasquier.
\newblock On some smooth projective two-orbit varieties with {P}icard number 1.
\newblock {\em Math. Ann.}, 344(4):963--987, 2009.

\bibitem{RichmondSlofstra}
Edward Richmond and William Slofstra.
\newblock Billey-{P}ostnikov decompositions and the fibre bundle structure of
  {S}chubert varieties.
\newblock {\em Math. Ann.}, 366(1-2):31--55, 2016.

\bibitem{Seshadri87}
Conjeeveram~S. Seshadri.
\newblock Line bundles on {S}chubert varieties.
\newblock In {\em Vector bundles on algebraic varieties ({B}ombay, 1984)},
  volume~11 of {\em Tata Inst. Fund. Res. Stud. Math.}, pages 499--528. Tata
  Inst. Fund. Res., Bombay, 1987.

\bibitem{Stembridge}
John.~R. Stembridge.
\newblock Multiplicity-free products and restrictions of {W}eyl characters.
\newblock {\em Represent. Theory}, 7:404--439, 2003.

\end{thebibliography}
\bibliographystyle{plain}

\end{document}